\documentclass[reqno]{amsart}
\usepackage{amsmath, amssymb, amsthm, epsfig}
\usepackage{hyperref, latexsym}
\usepackage{url}
\usepackage[mathscr]{euscript}

\usepackage{color}
\usepackage{fullpage} 
\usepackage{setspace}

\def\today{\ifcase\month\or
  January\or February\or March\or April\or May\or June\or=
  July\or August\or September\or October\or November\or December\fi
  \space\number\day, \number\year}

 \newtheorem{theorem}{Theorem}
  
 \newtheorem{lemma}[theorem]{Lemma}
 \newtheorem{proposition}[theorem]{Proposition}
 
 \theoremstyle{definition}

 \theoremstyle{remark}
 \newtheorem{remark}[theorem]{Remark}

 \newcommand{\R}{\mathbb{R}}
  
 \newcommand{\N}{\mathbb{N}}

\newcommand{\var}{{\rm Var\,}}
\newcommand{\intav}[1]{\mathchoice {\mathop{\vrule width 6pt height 3 pt depth  -2.5pt
\kern -8pt \intop}\nolimits_{\kern -6pt#1}} {\mathop{\vrule width
5pt height 3  pt depth -2.6pt \kern -6pt \intop}\nolimits_{#1}}
{\mathop{\vrule width 5pt height 3 pt depth -2.6pt \kern -6pt
\intop}\nolimits_{#1}} {\mathop{\vrule width 5pt height 3 pt depth
-2.6pt \kern -6pt \intop}\nolimits_{#1}}}

\begin{document}

\title[Sharp Poincare-Wirtinger inequalities on complete graphs]{Sharp Poincare-Wirtinger inequalities on complete graphs}
\subjclass[2010]{}

\author[C. Gonz\'alez-Riquelme]{Cristian Gonz\'alez-Riquelme}
\address{Center for Mathematical Analysis, Geometry and Dynamical Systems and Departamento de Matemática, Instituto Superior Técnico, Av. Rovisco Pais, 1049-001 Lisboa, Portugal}
\email{cristian.g.riquelme@tecnico.ulisboa.pt}

\author[J. Madrid]{Jos\'e Madrid}
\address{Department of  Mathematics, Virginia Polytechnic Institute and State University,  225 Stanger Street, Blacksburg, VA 24061-1026, USA}
\email{josemadrid@vt.edu}


\allowdisplaybreaks
\numberwithin{equation}{section}

\maketitle
\begin{abstract}
    Let $K_n=(V,E)$ be the complete graph with $n\geq 3$ vertices (here $V$ and $E$ denote the set of vertices and edges of $K_n$ respectively). We find the optimal value ${\bf{C}}_{n,p}$ such that the inequality 
$$\|f-m_f\|_p\le {\bf C}_{n,p}\var_{p}f$$
holds for every $f:V\to \mathbb{R},$ where $\var_p$ stands for the $p$-variation, and $m_f$ stands for the average value of $f$, for all $p\in[1,3+\delta^1_n)\cup (3+\delta^2_n,+\infty)$, for  $\delta^1_n=\frac{1}{2n^2\log(n)}+O(1/n^3)$ and $\delta^2_n=\frac{2}{n}+O(1/n^2).$ Moreover, we characterize all the maximizer functions in that case. The behavior of the maximizers is different in each of the intervals $(1,2)$, $(2,3+\delta^{1}_n)$ and $(3+\delta^{2}_n,\infty).$
\end{abstract}

\section{Introduction}
The study of Poincare inequalities is a central theme in analysis, in particular, they play a relevant role in partial differential equations and mathematical physics. In recent years, proving discrete analogues of classical analytic inequalities has attracted the attention of many authors, see for instance \cite{GM},\cite{GM2},\cite{ST1}. In particular, in \cite{IV} and \cite{ILVV} the optimal constant for Poincare inequalities on the hypercube was studied. In \cite{LSTV} the Poincare inequalities constant grow on graphs was studied. Poincare inequalities also appear naturally when studying isoperimetric inequalities \cite{BG}. In this paper, we establish optimal Poincare-Wirtinger type inequalities on complete graphs and we characterize the extremizers. That is the content of our main result. In the following, $K_n$ denotes the complete graph with $n$ vertices $a_1, a_2,\dots, a_n$. For all $f:K_n\to\R$ we denote by
$$
\|f\|_p=\left(\displaystyle\sum_{a_i\in K_n}|f(a_i)|^{p}\right)^{\frac{1}{p}}.
$$
the $\ell^p-$norm of $f$. Also, we denote by
$$
\var_p(f)=\left(\sum_{1\leq i<j\leq n}|f(a_i)-f(a_j)|^{p}\right)^{\frac{1}{p}},
$$
the $p-$variation of $f$.

\begin{theorem}[Main theorem]
    Let $n\in\N$, and let  $f:K_n\to \mathbb{R}$ and $m=m_f:=\frac{\sum_{a_i\in \mathbb{K}_n}f(a_i)}{n}$.
    \begin{itemize}
    \item[(i)] If $p=1$ and $n\geq 3$, then we have 
  \begin{align*}
  \sum_{a_i\in K_n}|f(a_i)-m|\leq\frac{2}{n}\var_1(f),   
 \end{align*}
Moreover, if we assume without loss of generality that $f(a_1)\geq\dots\geq f(a_k)\geq m\geq f(a_{k+1})\geq\dots\geq f(a_n)$,
the equality happens iff $f(a_i)=f(a_1)$ for all $1\leq i \leq k$ and $f(a_i)=f(a_n)$ for all $k+1\leq i \leq n$.
  \item[(ii)] If $1<p<2$ and $n\geq 3$, then we have 
  \begin{align*}
  \sum_{a_i\in K_n}|f(a_i)-m|^p\leq \frac{(\lfloor{\frac{n}{2}}\rfloor)^{p-1}+(\lceil{\frac{n}{2}}\rceil)^{p-1}}{n^{p}}\var_p(f)^p.   
 \end{align*}
 In this case, the maximizers are functions $f:K_n\to \mathbb{R}$ with $f(K_n)=\{A,B\}$ with $A>B$ and $|f^{-1}(A)|\in \{\lfloor{}\frac{n}{2}\rfloor{},\lceil{}\frac{n}{2}\rceil{}\}$.
  \item[(iii)] If $p=2$ and $n\geq 3$, then we have an identity
  \begin{align*}
  \sum_{a_i\in K_n}|f(a_i)-m|^2=\frac{1}{n}\var_2(f)^2.   
 \end{align*}
            \item [(iv)] If $2< p\leq 3$ and $n\geq 3$, then
 \begin{align*}
  \sum_{a_i\in K_n}|f(a_i)-m|^p\le \frac{1}{2^{p-1}+n-2}\var_p(f)^p.   
 \end{align*}
 and this inequality is sharp. The equality is attained for any function $f:K_n\to\R$ assuming values $a+c,\underbrace{a,\dots,a}_{n-2},a-c$, for $a,c\in\R$. Moreover, the sharp inequality can be extended to $p$ belonging to the interval $(3,3+\delta_n^{1})$, with the same maximizers. In this case, we take $$\delta_n^{1}=\frac{\log(\sqrt{n^2+4}+3n)-\log(4n)}{\log(n-1)}=\frac{1}{2n^2\log(n)}+O(1/n^3).$$
  \item [(v)] If $3<p< 4$ such that $\frac{n-2}{n}\ge \frac{1+(n-1)^{p-1}}{n^{p-1}}.$ In particular, if $$p\ge 3+\delta_n^2:=1+\frac{2\log(n)-\log(n^2-2n-1)}{\log(n)-\log(n-1)}=3+\frac{2}{n}+O(\frac{1}{n^2}).$$ Then 
 \begin{align}
 \sum_{a_i\in K_n}|f(a_i)-m|^p\le \frac{1+(n-1)^{p-1}}{n^{p}}\var_p(f)^p
 \end{align}
  and this inequality is sharp. The equality is attained for Dirac delta functions, i.e. functions with $f(K_n)=\{A,B\}$ with $A\neq B$ and $|f^{-1}(A)|=1.$ 
 \item [(vi)] If $p\geq 4$ and $n\geq 3$. Then 
 \begin{align}
 \sum_{a_i\in K_n}|f(a_i)-m|^p\le \frac{1+(n-1)^{p-1}}{n^{p}}\var_p(f)^p,
 \end{align}
 and this inequality is sharp. The equality is attained for Dirac delta functions. 
    \end{itemize}
\end{theorem}
The idea behind the proof of these results (for $p\neq 1,2$) can be outlined as follows. First, we prove that the maximizers for the inequality 
\begin{align}\label{sharppoincarewirtinger}
\|f-m_f\|_{p}\le {\bf C}_{n,p}\var_{p}f    
\end{align}
exist. 

Then, we observe that, being critical points, these maximizers should satisfy a functional equation that depends on the optimal constant ${\bf C}_{n,p}$. On the other hand, we prove that these functional equations are the equality case of a sharp inequality (when ${\bf C}_{n,p}$ is as conjectured). This step depends strongly on the range to which $p$ belongs, and is the one that determines the different structure of the maximizers in each interval. The proof of this auxiliary inequality (Lemma \ref{key lemma 2<p}) takes most of the efforts in this manuscript, and is based on symmetrization and concavity/convexity arguments. 
Therefore, by classifying the maximizers of this auxiliary inequality, we determine all the critical points of the original \eqref{sharppoincarewirtinger} inequality. From this we conclude the classification of the maximizers and we obtain the optimal constant ${\bf C}_{n,p}$.         
\section{Proof of Main results}
We start with a basic lemma.
\begin{lemma}[There is an extremizer]\label{existence of extremizers}
There is a function $f:K_n\to\R$ such that
$$
\sup_{g:K_n\to\R, \text{non constant}}\frac{\sum_{a_i\in K_n}|g(a_i)-m|^p}{\var_p(g)^p}=\sup_{g:K_n\to[0,1], \text{non constant}}\frac{\sum_{a_i\in K_n}|g(a_i)-m|^p}{\var_p(g)^p}=\frac{\sum_{a_i\in K_n}|f(a_i)-m|^p}{\var_p(f)^p}.
$$
\end{lemma}
\begin{proof}
The first equality follows from the fact that the quotient is translation and dilation invariant, more precisely, the quotient is preserved by the transformation
$$
g\mapsto \frac{g-\min_{i}g(a_i)}{\max_{j=1,\dots,n}[g(a_j)-\min_{i}g(a_i)]}.
$$
 Then, given $y:=(y_1,\dots,y_n)\in A:=\{(y_1,y_2,\dots,y_n)\in[0,1]^{n};\underset{{i=1,\dots,n}}{\max} y_{i}=1, \underset{{i=1,\dots,n}}{\min} y_{i}=0\}$ we define $f_y:V\to \mathbb{R}_{\ge 0}$ by $f_y(a_i):=y_i$. We observe that $\frac{\sum_{a_i\in K_n}|f_y(a_i)-m_{f_y}|^p}{\var_p(f_y)^p}$ is continuous with respect to $y$ in the compact set $A$. Thus, it achieves its maximum at some point $y_0\in A$. Then $$\sup_{g:K_n\to\R, \text{non constant}}\frac{\sum_{a_i\in K_n}|g(a_i)-m|^p}{\var_p(g)^p}=\frac{\sum_{a_i\in K_n}|f_{y_0}(a_i)-m|^p}{\var_p(f_{y_0})^p}.$$ 
From where we obtain the result. 
\end{proof}

{\it{Proof of Main theorem:}}

{\it{Case: p=1}}. In this case, without loss of generality we assume that $f(a_1)\geq\dots\geq f(a_k)\geq m\geq f(a_{k+1})\geq\dots\geq f(a_n)$. Then, we have
\begin{align*}
\sum_{i=1}^{n}|f(a_i)-m|&=\sum_{i=1}^{k}(f(a_i)-m)+\sum_{i=k+1}^{n}(m-f(a_i))\\
&=\sum_{i=1}^{k}f(a_i)-\sum_{i=k+1}^{n}f(a_i)+\frac{(n-2k)}{n}\sum_{i=1}^{n}f(a_i)\\
&=\frac{2}{n}\left[(n-k)\sum_{i=1}^{k}f(a_i)-k\sum_{i=k+1}^{n}f(a_i)\right]\\
&=\frac{2}{n}\sum_{1\leq i\leq k, k+1\leq j\leq n}f(a_i)-f(a_j)\\
&\leq \frac{2}{n}\sum_{i<j}f(a_i)-f(a_j).
\end{align*}
In the last inequality the identity is attained iff $f(a_i)=f(a_1)$ for all $1\leq i \leq k$ and $f(a_i)=f(a_n)$ for all $k+1\leq i \leq n$.

{\it{Case $p=2$}}. In this case we have an identity
 \begin{align*}
  \sum_{i=1}^{n}(f(a_i)-m)^2&=
   \sum_{i=1}^{n}[f(a_i)^2+m^2-2mf(a_i)]\\
   &=\left[\sum_{i=1}^{n}f(a_i)^2\right]+nm^2-2nm^2
   =\left[\sum_{i=1}^{n}f(a_i)^2\right]-nm^2\\
   &=\sum_{i=1}^{n}f(a_i)^2-\frac{1}{n}\sum_{i=1}^{n}f(a_i)^2-\frac{2}{n}\sum_{i<j}f(a_i)f(a_j)\\
   &=\frac{n-1}{n}\sum_{i=1}^{n}f(a_i)^2-\frac{2}{n}\sum_{i<j}f(a_i)f(a_j)
   \\
   &= \frac{1}{n}\sum_{i<j}f(a_i)^2+f(a_j)^2-2f(a_i)f(a_j)\\
  &=\frac{1}{n}\var_2(f)^2.    
 \end{align*}
Thus, any function $f:K_n\to\R$ is an extremizer.\\

Now we discuss the case $2\neq p>1$. Let
\begin{equation}\label{sup constant}
C_n=C_{n,p}:=\sup_{g:K_n\to\R} \frac{\sum_{a_i\in K_n}|g(a_i)-m|^p}{\var_p(g)^p}.
\end{equation}
By taking $g_0$ such that $g_0(a_1)=1, g_0(a_n)=-1$ and $g_0(a_i)=0$ for $i=2,\dots,n-1$ we observe that $C_n\ge \frac{1}{2^{p-1}+n-2}.$
By taking $g_1$ such that $g_1(a_1)=1,$ and $g_1(a_i)=0$ for $i=2,\dots,n$ we observe that $C_n\ge \frac{1+(n-1)^{p-1}}{n^{p}}$. Moreover, by taking $g_2(a_i)=1$ for $i\le \lfloor{}\frac{n}{2}\rfloor{}$ and $g_2(a_i)=0$ for $i>\lfloor{}\frac{n}{2}\rfloor{}$, we also obtain that $C_n\ge \frac{ \lfloor{}\frac{n}{2}\rfloor{}^{p-1}+ \lceil{}\frac{n}{2}\rceil{}^{p-1}}{n^p}.$ 
Then
\begin{equation}\label{max constant}
C_n\geq \max\left\{\frac{1}{2^{p-1}+n-2},\frac{1+(n-1)^{p-1}}{n^{p}}, \frac{ \lfloor{}\frac{n}{2}\rfloor{}^{p-1}+ \lceil{}\frac{n}{2}\rceil{}^{p-1}}{n^p}\right\},
\end{equation}
By the previous lemma, there is an extremizer function $f$ for \eqref{sup constant}. We assume without loss of generality that $f(a_1)\ge f(a_2)\ge \dots \ge f(a_n).$

For $\epsilon>0$ we define the auxiliary function $f_{\varepsilon}$ by: $f_{\varepsilon}(a_i):=f(a_i)$ for $i=2,\dots,n-1$ and $f_{\varepsilon}(a_1):=f(a_1)+\varepsilon$ and $f_{\varepsilon}(a_n):=f(a_n)-\varepsilon$. Observe that $m(f)=m(f_{\varepsilon})$.
We consider the function
$$r(\varepsilon):=C_n\var_p(f_{\varepsilon})^p-\sum_{a_i\in K_n}|f_{\varepsilon}(a_i)-m|^p.$$

Observe that
\begin{equation}\label{derivative=0}
0=\frac{r'(0)}{p}=C_n\left(\sum_{a_i\in K_n}(f(a_1)-f(a_i))^{p-1}+(f(a_i)-f(a_n))^{p-1}\right)-(f(a_1)-m)^{p-1}-(m-f(a_n))^{p-1}.
\end{equation}

Then, we observe the following:

\begin{lemma}\label{key lemma 2<p}
For $p\in(1,2)\cup(2,3+\delta^1_n]\cup[(3,4)\cap \mathcal{A}_n]\cup[4,+\infty)$ and $C^*_{n,p}$ be such that 

 \begin{equation}\label{Cnp*}
    C^*_{n,p}=\left\{
\begin{array}{ll}
\frac{\lfloor{}\frac{n}{2}\rfloor{}^{p-1}+\lceil{}\frac{n}{2}\rceil{}^{p-1}}{n^{p}}, &\text{ if } p\in (1,2)\\
 \frac{1+(n-1)^{p-1}}{n^{p}}, & \text{ if } p\in \left((3,4]\cap \mathcal{A}_n\right)\cup [4,\infty),\\
\frac{1}{2^{p-1}+n-2}, & \text{ if } p\in (2,3+\delta^1_n).

\end{array} \right.
    \end{equation}
where $\delta^1_n$ is defined in Theorem 1 and $\mathcal{A}_n:=\{x>0;\frac{n-2}{n}\ge \frac{1+(n-1)^{x-1}}{n^{x-1}}\}.$
Let $f:K_n\to\R$, assuming that $f(a_1)=\max_{i=1}^{n} \{f(a_i)\}$ and $f(a_n)=\min_{i=1}^{n}\{f(a_n)\}$, and $m=m_f:=\frac{1}{n}\sum_{i=1}^{n}f(a_i)$, then
\begin{align}\label{eq: key lemma 2<p}
 (f(a_1)-m)^{p-1}+(m-f(a_n))^{p-1}\le C_{n,p}^*\left(\sum_{j=1}^{n}(f(a_1)-f(a_j))^{p-1}+(f(a_j)-f(a_n))^{p-1}\right),
\end{align}
and the constant $C_{n,p}^{*}$ can not be replaced by a smaller quantity. Moreover, we have that the maximizers satisfy the following property (when $f(a_1)\ge f(a_2)\ge \dots \ge f(a_n)$): 
\begin{itemize}
\item[(i)] $f(a_1)=f(a_j)$ if $j\le \lfloor{}\frac{n}{2}\rfloor{}$ and $f(a_j)=f(a_n)$ otherwise; or $f(a_1)=f(a_j)$ if $j\le \lceil{} \frac{n}{2}\rceil{}$ and $f(a_j)=f(a_n)$ otherwise, for $1<p<2.$ 
\item[(ii)] $f(a_1)>f(a_2)=f(a_3)=\dots= f(a_{n-1})>f(a_n)$ and $f(a_2)=\frac{f(a_1)+f(a_n)}{2}$, if $2<p\leq 3+\delta_n^1$.
\item[(iii)] $f(a_1)>f(a_2)=\dots=f(a_n)$ or $f(a_1)=\dots f(a_{n-1})>f(a_n)$ if $4\leq p$, $p\in (3,4)\cap
\mathcal{A}_n$.
\end{itemize}

\end{lemma}
\begin{remark}
Since $\frac{1+(n-1)^x}{n^x}$ is decreasing in $x$ for any fixed $n$, we have that $\mathcal{A}_n$ is an interval for all $n\geq 3$. We observe that for $n\ge 3$ we have $4\in \mathcal{A}_n$. And, moreover, $1+\log_{\frac{n}{n-1}}\frac{n^2}{n^2-2n-1}\in \mathcal{A}_n$ since $\frac{1+(n-1)^{p-1}}{n^{p-1}}\leq \frac{1}{n^2}+\left(\frac{n-1}{n}\right)^{p-1}$ and $\frac{1}{n^2}+\left(\frac{n-1}{n}\right)^{p-1}=\frac{n-2}{n}$ for $p=1+\log_{\frac{n}{n-1}}\frac{n^2}{n^2-2n-1}$. Therefore, since $\lim_{n\to\infty}1+\log_{\frac{n}{n-1}}\frac{n^2}{n^2-2n-1}=3$, we have that for any $p>3$ there exists $n_0$ such that for any $n>n_0$ we have $p\in \mathcal{A}_n$.
\end{remark}
Now we explain how Lemma \ref{key lemma 2<p} implies the remaining cases of Theorem 1. In fact, given that we have \eqref{derivative=0} for any maximizer of \eqref{sup constant}, by using the Lemma \ref{key lemma 2<p} we have $C_{n,p}^*\ge C_{n,p}$. On the other hand, by \eqref{Cnp*} and \eqref{max constant}, $$C_{n,p}^*\le \max\left\{\frac{1}{2^{p-1}+n-2},\frac{1+(n-1)^{p-1}}{n^{p}}, \frac{\lfloor{}\frac{n}{2}\rfloor{}^{p-1}+\lceil{}\frac{n}{2}\rceil{}^{p-1}}{n^{p}}\frac{}{}\right\}\leq C_{n,p}.$$ Thus, we conclude that \begin{align}\label{equalityofC}
C_{n,p}^*=C_{n,p}, 
\end{align}
as desired. Furthermore, in Lemma \ref{key lemma 2<p} we characterize the maximizers for \eqref{eq: key lemma 2<p}. And, by \eqref{sup constant} and \eqref{equalityofC} the maximizers of \eqref{sup constant} are the same as for \eqref{eq: key lemma 2<p}.

Now, we discuss the proof of Lemma \ref{key lemma 2<p}.

\begin{proof}[Proof of Lemma \ref{key lemma 2<p}]

Let $R_{n,p}$ be the best constant for the inequality \eqref{eq: key lemma 2<p}. By considering the functions $g_0$, $g_1$ and $g_2$ as before, we observe that $R_{n,p}\ge \max\left\{\frac{1}{2^{p-1}+n-2},\frac{1+(n-1)^{p-1}}{n^{p}}, \frac{\lfloor{}\frac{n}{2}\rfloor{}^{p-1}+\lceil{}\frac{n}{2}\rceil{}^{p-1}}{n^{p}} \right\}$. 
Let $f$ be an extremizer for \eqref{eq: key lemma 2<p} (the existence of an extremizer can be seen similarly to Lemma \ref{existence of extremizers}). Now, we analyze the different cases of the lemma.

{\it Case: $1<p<2$:} We observe that, in this case, the function $x\mapsto x^{p-1}$ is concave for $x\ge 0$. Therefore, we observe by Karamata's inequality, that if $a>b>0$, and $0\leq \varepsilon<\frac{a-b}{2}$, then 
\begin{equation}\label{eq: a b epsilon}
(a-\varepsilon)^{p-1}+(b+\varepsilon)^{p-1}\ge a^{p-1}+b^{p-1}.
\end{equation} If there exist $a_j\neq a_k$ with $f(a_1)>f(a_j)\ge f(a_k)>f(a_n)$, then for $\varepsilon>0$ sufficiently small, defining  $f_{\varepsilon}:K_n\to\R$ by $f_{\varepsilon}(a_j)=f(a_j)+\varepsilon,$ $f_{\varepsilon}=f(a_k)-\varepsilon$, and $f_{\varepsilon}=f$ otherwise. We have that, since
$$
(f_{\varepsilon}(a_1)-f_{\varepsilon}(a_k)-\varepsilon)^{p-1}+(f_{\varepsilon}(a_1)-f_{\varepsilon}(a_j)+\varepsilon)^{p-1}> (f_{\varepsilon}(a_1)-f_{\varepsilon}(a_k))^{p-1}+(f_{\varepsilon}(a_1)-f_{\varepsilon}(a_j))^{p-1}.
$$
by taking $a=f_{\varepsilon}(a_1)-f_{\varepsilon}(a_k)$ and $b=f_{\varepsilon}(a_1)-f_{\varepsilon}(a_j)$ in \eqref{eq: a b epsilon}. And, similarly
$$
(f_{\varepsilon}(a_j)-f_{\varepsilon}(a_n)-\varepsilon)^{p-1}+(f_{\varepsilon}(a_k)-f_{\varepsilon}(a_n)+\varepsilon)^{p-1}> (f_{\varepsilon}(a_j)-f_{\varepsilon}(a_n))^{p-1}+(f_{\varepsilon}(a_k)-f_{\varepsilon}(a_n))^{p-1}.
$$
by taking $a=f_{\varepsilon}(a_j)-f_{\varepsilon}(a_n)$ and $b=f_{\varepsilon}(a_k)-f_{\varepsilon}(a_n)$ in \eqref{eq: a b epsilon}. Then
\begin{align*}\frac{(f(a_1)-m)^{p-1}+(m-f(a_n))^{p-1}}{\left(\sum_{j=1}^{n}(f(a_1)-f(a_j))^{p-1}+(f(a_j)-f(a_n))^{p-1}\right)}<\frac{(f_{\varepsilon}(a_1)-m)^{p-1}+(m-f_{\varepsilon}(a_n))^{p-1}}{\left(\sum_{j=1}^{n}(f_{\varepsilon}(a_1)-f_{\varepsilon}(a_j))^{p-1}+(f_{\varepsilon}(a_j)-f_{\varepsilon}(a_n))^{p-1}\right)}.
\end{align*}
Thus, we would have that $f$ is not a maximizer of our inequality. This, implies that this $a_j$ and $a_k$ do not exist. Thus, $f$ is of the form $f(a_1)=\dots=f(a_{j-1})\ge f(a_j)>f(a_{j+1})=\dots=f(a_n).$ By the invariance, we can assume that $f(a_1)=1$ and $f(a_n)=0$. Now, we deal with the case when $1>x>0.$ Here, $m=\frac{j+x}{n},$ and 
\begin{align*}&\frac{(f(a_1)-m)^{p-1}+(m-f(a_n))^{p-1}}{\left(\sum_{j=1}^{n}(f(a_1)-f(a_j))^{p-1}+(f(a_j)-f(a_n))^{p-1}\right)}\\
 &=\frac{1}{n^{p-1}}\frac{\left(n-j-x\right)^{p-1}+\left(j+x\right)^{p-1}}{(1-x)^{p-1}+x^{p-1}+n-1}.
 \end{align*}
 We will observe that this is bounded from above by 
 $$ \frac{\lfloor{}\frac{n}{2}\rfloor{}^{p-1}+\lceil{}\frac{n}{2}\rceil{}^{p-1}}{n^{p}}.$$
 The case $n$ even follows using that $(1-x)^{p-1}+x^{p-1}\ge 1$, and since $\max\{n-j-x,j+x\}\ge \frac{n}{2}\ge \min\{n-j-x,j+x\},$ by Karamata's inequality $\left(n-j-x\right)^{p-1}+\left(j+x\right)^{p-1}\leq \lfloor{}\frac{n}{2}\rfloor{}^{p-1}+\lceil{}\frac{n}{2}\rceil^{p-1}.$
For the  odd case $n=2k+1,$ we observe that if $j\notin \{k,k+1\}$, then $\max\{n-j-x,j+x\}\ge k+1\ge k\ge \min\{n-j-x,j+x\},$ therefore by the concavity of $x\mapsto x^{p-1}$ we get $$\left(n-j-x\right)^{p-1}+\left(j+x\right)^{p-1}\le k^{p-1}+(k+1)^{p-1},$$ from where the result follows using once again that $(1-x)^{p-1}+x^{p-1}\ge 1$. 

Now, we discuss the case when $j\in \{k,k+1\}.$ However, if $j=k+1$ we get that $(n-j-x,j+x)=(k-x,k+1+x)$, then by concavity $(k+1+x)^{p-1}+(k-x)^{p-1}\le (k+1-x)^{p-1}+(k+x)^{p-1},$ that corresponds to the case  when $k=j$. Therefore, we just need to prove that 
\begin{align*}
 \frac{1}{(2k+1)^{p-1}}\frac{\left(k+1-x\right)^{p-1}+\left(k+x\right)^{p-1}}{(1-x)^{p-1}+x^{p-1}+2k}\le \frac{k^{p-1}+(k+1)^{p-1}}{(2k+1)^{p}},   
\end{align*}
or, equivalently, 
\begin{align}\label{expressionp<2simple}
(2k+1)\left((k+1-x)^{p-1}+(k+x)^{p-1}\right)\le (k^{p-1}+(k+1)^{p-1})((1-x)^{p-1}+x^{p-1}+2k).
\end{align}
Writing $x=\frac{1}{2}+u$ for $u\in[-1/2,1/2]$, the previous inequality reduced to prove that
\begin{align*}
&(2k+1)\left(k+\frac{1}{2}\right)^{p-1}\left[\left(1-\frac{u}{k+1/2}\right)^{p-1}+\left(1+\frac{u}{k+1/2}\right)^{p-1}\right]\\
&\leq (k^{p-1}+(k+1)^{p-1})\left[2k+\frac{1}{2^{p-1}}((1-2u)^{p-1}+(1+2u)^{p-1})\right].
\end{align*}
Using the expansion $(1+y)^{\alpha}=\sum_{l=0}^{\infty}\binom{\alpha}{l}y^{l}$ (converging absolutely for $-1<y<1$), this is equivalent to 
\begin{align*}
2(2k+1)\left(k+\frac{1}{2}\right)^{p-1}\sum_{l=0}^{\infty}\binom{p-1}{2l}\left(\frac{u}{k+\frac{1}{2}}\right)^{2l}\le (k^{p-1}+(k+1)^{p-1})\left(2k+\frac{2}{2^{p-1}}\left(\sum_{l=0}^{\infty}\binom{p-1}{2l}(2u)^{2l}\right)\right),
\end{align*}
or, equivalently,

\begin{align}\label{expressionp<2}
\begin{split}
2(2k+1)\left(k+\frac{1}{2}\right)^{p-1}&-(k^{p-1}+(k+1)^{p-1})\left(2k+\frac{2}{2^{p-1}}\right)\\
&\le \sum_{l=1}^{\infty}\binom{p-1}{2l}\left(2^{2l+2-p}\left((k^{p-1}+(k+1)^{p-1})-(2k+1)^{p-2l}\right)\right)u^{2l}.
\end{split}
\end{align}
We observe that $\binom{p-1}{2l}<0$ and $k^{p-1}+(k+1)^{p-1}\ge (2k+1)^{p-1}$ since $1<p<2.$ Thus, $$\binom{p-1}{2l}\left(2^{2l+2-p}\left((k^{p-1}+(k+1)^{p-1})-(2k+1)^{p-2l}\right)\right)<0,$$ for any $l>0.$ Therefore, the RHS of \eqref{expressionp<2} has only negative coefficients, and thus it is a decreasing function for $u$ in $[0,\frac{1}{2}]$ and it is an increasing function for $u$ in $[-\frac{1}{2},0]$. Then, inequality \eqref{expressionp<2} holds if it holds for $u=\frac{1}{2}$ and $u=-\frac{1}{2}$, that follows directly by going back to the expression \eqref{expressionp<2simple}, in fact, we observe that for $x=1$, and $x=0$ we have an identity in \eqref{expressionp<2simple}. From where we conclude the case $1<p<2$.
    
Now, for $p>2$ we observe the following. Let us consider the function $g:K_n\to\R$ defined by
$g(a_1):=f(a_1)$, $g(a_n):=f(a_n)$ and $g(a_i)=\frac{1}{n-2}\sum_{i=2}^{n-1}f(a_i)$. Then, since $m_g=m_f$, we have
\begin{equation*}
  (g(a_1)-m_g)^{p-1}+(m_g-g(a_n))^{p-1}  =(f(a_1)-m_f)^{p-1}+(m_f-f(a_n))^{p-1}.
\end{equation*}
On the other hand, since the function $h_1:[0,+\infty)\to\R$ defined by $h_1(x)=x^{p-1}$ is convex for $p\geq 2$, by Karamata's inequality we have
\begin{align*}
\sum_{j=1}^{n}(f(a_1)-f(a_j))^{p-1}+(f(a_j)-f(a_n))^{p-1}\geq \sum_{j=1}^{n}(g(a_1)-g(a_j))^{p-1}+(g(a_j)-g(a_n))^{p-1}.
\end{align*}
Therefore
\begin{equation}
\frac{|g(a_1)-m|^{p-1}+|m-g(a_n)|^{p-1}}{\sum_{j=1}^{n}(g(a_1)-g(a_j))^{p-1}+(g(a_j)-g(a_n))^{p-1}}\geq \frac{|f(a_1)-m|^p+|m-f(a_n)|^{p-1}}{\sum_{j=1}^{n}(f(a_1)-f(a_j))^{p-1}+(f(a_j)-f(a_n))^{p-1}}.
\end{equation}

Thus, we can assume without loss of generality that $f(a_1)\geq f(a_2)=f(a_3)=\dots=f(a_{n-1})\geq f(a_n)$. Since $f$ is not constant then, we have three scenarios
\begin{itemize}
\item $f(a_1)> f(a_2)=f(a_3)=\dots=f(a_{n-1})> f(a_n)$.
\item $f(a_1)> f(a_2)=f(a_3)=\dots=f(a_{n-1})= f(a_n)$.
\item $f(a_1)= f(a_2)=f(a_3)=\dots=f(a_{n-1})> f(a_n)$.
\end{itemize}
Since the quotients we are considering are translation and dilation invariant. In the first scenario, we can assume without loss of generality that $f(a_1)=1$, $f(a_n)=-1$ and $f(a_2)=\dots=f(a_{n-1})=x\ge 0$. On the other hand, in the second and third scenarios we can assume without loss of generality that $f(a_1)=1$ and $f(a_2)=\dots=f(a_{n-1})=f(a_n)=0$, having a Dirac delta function.

Then, we are left to prove the following inequality 
\begin{align*}
\left(1-\frac{n-2}{n}x\right)^{p-1}+\left(1+\frac{n-2}{n}x\right)^{p-1}\le C_{n,p}^{*}\left(2^{p}+(n-2)((1-x)^{p-1}+(1+x)^{p-1})\right).
\end{align*}
Defining the function $T:[0,1]\to\R$ by
$$
T(x):=C_{n,p}^{*}\left(2^{p}+(n-2)((1-x)^{p-1}+(1+x)^{p-1})\right)-\left(1-\frac{n-2}{n}x\right)^{p-1}-\left(1+\frac{n-2}{n}x\right)^{p-1}.
$$
We want to show that $T(x)\geq 0$ for all $x\in[0,1]$. We will prove that $T(0),T(1)\ge 0$ and that $T$ is monotone in $[0,1]$, we discuss each case separately. 

{\it{Case: $2<p\leq 3+\delta_{n}^1$.}} In this case,  we observe that $T(0)=0$, and we will prove that $T$ is increasing in $[0,1]$. For this is enough to prove that $T'(x)\geq 0$ for all $x\in[0,1]$, which is equivalent to 
\begin{align*}
nC_{n,p}^{*}\ge \frac{\left(1+\frac{n-2}{n}x\right)^{p-2}-(1-\frac{n-2}{n}x)^{p-2}}{(1+x)^{p-2}-(1-x)^{p-2}}:=G(x)\ \ \text{for all} \ \ x\in[0,1].
\end{align*}
{\it{Subcase $2<p\le 3$:}}
We observe that for a fix $x\in[0,1]$ the function $v_x:\left[\frac{n-2}{n},1\right]\to\R$ defined by $$v_x(y):=\left(y+\frac{n-2}{n}x\right)^{p-2}-\left(y-\frac{n-2}{n}x\right)^{p-2}$$
is decreasing in $[\frac{n-2}{n},1]$ as a function in $y$, since its derivative is $v'_x(y)=\left(y+\frac{n-2}{n}x\right)^{p-3}-\left(y-\frac{n-2}{n}x\right)^{p-3}\le 0$. Therefore, 
\begin{equation}\label{eq: cotas para 2<p<3 I}
\frac{\left(1+\frac{n-2}{n}x\right)^{p-2}-(1-\frac{n-2}{n}x)^{p-2}}{(1+x)^{p-2}-(1-x)^{p-2}}= \left(\frac{n-2}{n}\right)^{p-2}\frac{v_x(1)}{v_x(\frac{n-2}{n})}\le \left(\frac{n-2}{n}\right)^{p-2}.
\end{equation}
Moreover, since the function $r:\R\to\R$ defined by $r(p):=p-1-2^{p-2}$ is concave on $[2,3]$, then $r(p)\geq \min\{r(2),r(3)\}=0$ for all $p\in[2,3]$. Then
$$
2^{p-1}(n-2)^{p-2}\leq 2(p-1)(n-2)^{p-2}.
$$
Moreover, by the mean value theorem
$(p-1)(n-2)^{p-2}\leq \frac{n^{p-1}-(n-2)^{p-1}}{2} $. Then $2^{p-1}(n-2)^{p-2}\leq n^{p-1}-(n-2)^{p-1}$ or equivalently
\begin{equation}\label{eq: cotas para 2<p<3 II}
\left(\frac{n-2}{n}\right)^{p-2}\le \frac{n}{n-2+2^{p-1}}.
\end{equation}
Combining \eqref{eq: cotas para 2<p<3 I} and \eqref{eq: cotas para 2<p<3 II} we obtain
\begin{equation}\label{eq: cotas para 2<p<3 III}
\frac{\left(1+\frac{n-2}{n}x\right)^{p-2}-(1-\frac{n-2}{n}x)^{p-2}}{(1+x)^{p-2}-(1-x)^{p-2}} \le \frac{n}{n-2+2^{p-1}}= nC_{n,p}^{*}.
\end{equation}
Then, $T$ is increasing in $[0,1]$, as claimed. 

{\it{Subcase $3<p\le 3+\delta_n^{1}$:}} We observe that in this range, the function $x\mapsto x^{p-3}$ is strictly concave in $[1,n]$. Therefore, we have that
$$\frac{n-2}{n-1}n^{p-3}+\frac{1}{n-1}1^{p-3} < \left[\frac{n-2}{n-1}n+\frac{1}{n-1}\right]^{p-3}=(n-1)^{p-3},$$
 thus $(n-2)n^{p-3}+1< (n-1)^{p-2}$.
We notice that this implies that $$\lim_{x\to 0}G(x)=\frac{n-2}{n}<  \frac{(n-1)^{p-2}-1}{n^{p-2}}=G(1).$$
We also observe, that $G(1)\le nC_{n,p}^{*}=\frac{n}{n-2+2^{p-1}}.$ In fact, we observe that the inequality $$(n-2+4s)(n+2+4s)\le n^2,$$
is valid whenever $\frac{-\sqrt{n^2+4}-n}{4}\le s\le \frac{\sqrt{n^2+4}-n}{4}.$ So, we have that
$$
\max\left\{\frac{((n-1)^{p-3}-1)(n-1)}{4},2^{p-3}-1\right\}\leq \frac{\sqrt{n^2+4}-n}{4}.
$$
if $(n-1)^{p-3}-1\le \frac{\sqrt{n^2+4}-n}{4(n)}$. Here, we remark that this condition can be slightly improved, in particular, by replacing $4n$ by $n-1$ in the denominator for all $n\geq 4$, and, replacing $12$ by $4$ in the denominator for $n=3$. Then
\begin{align*}
((n-1)^{p-2}-1)(2^{p-1}+n-2)=(((n-1)^{p-3}-1)(n-1)+(n-2))((n+2)+4(2^{p-3}-1))\le n^{2}\le n^{p-1},
\end{align*}
which implies $G(1)\le nC_{n,p}^{*}$. Therefore, we need $(n-1)^{p-3}-1\le \frac{\sqrt{n^2+4}-n}{4(n)},$ or equivalently $$p\le 3+\frac{\log(\sqrt{n^2+4}+3n)-\log(4n)}{\log(n-1)}=3+\delta_n^{1}.$$ 
Now, we need to prove that $(1+\frac{n-2}{n}x)^{p-2}-(1-\frac{n-2}{n}x)^{p-2}\le nC_{n,p}^{*}((1+x)^{p-2}-(1-x)^{p-2}),$ or equivalently
\begin{align}\label{series33+}
\sum_{k=0}^{\infty}\binom{p-2}{2k+1}\left(nC_{n,p}^{*}-\left(\frac{n-2}{n}\right)^{2k+1}\right)x^{2k+1}\ge 0,
\end{align}
for all $x\in [0,1]$.
Since $\binom{p-2}{1}>0$ and $nC^{*}_{n,p}\geq G(1)> \frac{n-2}{n}$ as seen above, we have that the first coefficient is positive. On the other hand, for $k\ge 1$ we have that $\binom{p-2}{2k+1}<0$ and $nC_{n,p}^{*}-\left(\frac{n-2}{n}\right)^{2k+1}\ge 0$, thus, all the other coefficients in the left hand side of \eqref{series33+} are negative. Let us consider the polynomials $$P_N(x):=\sum_{k=0}^{N}\binom{p-2}{2k+1}\left(nC_{n,p}^{*}-\left(\frac{n-2}{n}\right)^{2k+1}\right)x^{2k+1}.$$
We know that \eqref{series33+} holds for $x=1$, since $nC^{*}_{n,p}\geq G(1)$. Then $P_{N}(1)>0$ for $N$ sufficiently large, and $P'_N(0)>0$ (since the first coefficient in the left hand side of \eqref{series33+} is positive). Moreover, we observe that since the higher degrees coefficients are negative $\lim_{x\to \infty}P_N(x)=-\infty$. Thus, for $N$ sufficiently large, $P_N(x)$ has a root between $1$ an $\infty$. However, by the Descartes's rule of signs, $P_N(x)$ has only one positive root (since it's coefficients list has only one sign change), therefore $P_N(x)\ge 0$ for all $x\in [0,1]$ for $N$ sufficiently large. Since $$P_N(x)\to \sum_{k=0}^{\infty}\binom{p-2}{2k+1}\left(nC_{n,p}^{*}-\left(\frac{n-2}{n}\right)^{2k+1}\right)x^{2k+1},$$
pointwise in $0<x<1,$ we conclude that \eqref{series33+} holds, and then our main theorem holds for $p\in (3,3+\delta_n^{1})$.

{\it{Cases: $4<p$ and $p\in (3,4)\cap A_n$}}

In this case we prove that $T$ is decreasing in $[0,1]$, then $T(x)\geq T(1)=0$ for all $x\in[0,1]$ and $R_{n,p}=\frac{1+(n-1)^{p-1}}{n^p}:=C_{n,p}^{*}$. For this is enough to prove that  
\begin{equation}\label{eq: 4<p ineq I}
\frac{(1+\frac{n-2}{n}x)^{p-2}-(1-\frac{n-2}{n}x)^{p-2}}{(1+x)^{p-2}-(1-x)^{p-2}}\ge \frac{1+(n-1)^{p-1}}{n^{p-1}},
\end{equation}
since this implies that $T'(x)\leq 0$ for all $x\in[0,1]$.

We study first the case $p\in  [4,\infty]$.
For this, we use the following proposition.
\begin{proposition}\label{prop: for p>4}
Let $r\ge 2$ and $k$ be a positive integer. Then, for any $y\ge 1$ we have
$$(1+ky)^r-1-(ky)^r\ge (k+y)^r-k^r-y^r.$$
\end{proposition}
\begin{proof}
We observe that the statement is true for $y=1$ (in fact, we have an identity). The result follows from the fact that $$h(y):=(1+ky)^r+k^r+y^r-1-(k+y)^r-(ky)^r$$ is increasing. Since, its derivative is $$h'(y)=kr(1+ky)^{r-1}+ry^{r-1}-r(k+y)^{r-1}-kr(ky)^{r-1},$$
and $k(1+ky)^{r-1}+y^{r-1}\ge (k+y)^{r-1}+k(ky)^{r-1},$
by Karamata's inequality, since the $(k+1)$-tuple $(1+ky,\dots,1+ky,y) $ majorizes $(k+y,ky, \dots,ky)$ or $(ky,\dots,ky, k+y)$ and $x\mapsto x^{r-1}$ is convex. 
\end{proof}
Then, using that
$$
\frac{(n-1)^{p-2}-1}{n^{p-2}}\ge \frac{1+(n-1)^{p-1}}{n^{p-1}},
$$
for all $p\geq 4, n\geq 3$ (since $(n-1)^{p-2}\geq (n-1)^2\geq n+1$). In order to have \eqref{eq: 4<p ineq I}, we just need to prove
\begin{align}
 \frac{(1+\frac{n-2}{n}x)^{p-2}-(1-\frac{n-2}{n}x)^{p-2}}{(1+x)^{p-2}-(1-x)^{p-2}}\ge \frac{(n-1)^{p-2}-1}{n^{p-2}}.   
\end{align}
This is can be rewritten as follows
\begin{align*}
 &(n+(n-2)x)^{p-2}+((n-1)(1-x))^{p-2}+(1+x)^{p-2}\\
 &\ge (n-(n-2)x)^{p-2}+((n-1)(1+x))^{p-2}+(1-x)^{p-2}.   
\end{align*}
Or, equivalently
\begin{align*}
 &\left((n-1)\frac{1+x}{1-x}+1\right)^{p-2}+(n-1)^{p-2}+\left(\frac{1+x}{1-x}\right)^{p-2}\\
 &\ge \left((n-1)+\frac{1+x}{1-x}\right)^{p-2}+\left((n-1)\frac{1+x}{1-x}\right)^{p-2}+1,  \end{align*}
which follows as a consequence of our Proposition \ref{prop: for p>4} by taking $r=p-2,$ $k=n-1$ and $y=\frac{1+x}{1-x}.$\\

Now we study the case when $p\in \mathcal{A}_n\cap [3,4)$, by the definition of $\mathcal{A}_n$, similarly to \eqref{eq: 4<p ineq I} it is enough to prove that
\begin{align}
\frac{(1+\frac{n-2}{n}x)^{p-2}-(1-\frac{n-2}{n}x)^{p-2}}{(1+x)^{p-2}-(1-x)^{p-2}}\ge \frac{n-2}{n}.  
\end{align}
Or, equivalently 
\begin{align*}
n\left(1+\frac{n-2}{n}x\right)^{p-2}-n\left(1-\frac{n-2}{n}x\right)^{p-2}-(n-2)(1+x)^{p-2}+(n-2)(1-x)^{p-2}:=\alpha(x)\ge 0.   
\end{align*}
We notice that $\alpha(0)=0$. Moreover, we have that 
\begin{align*}
\alpha'(x)=(p-2)(n-2)\left[\left(1+\frac{n-2}{n}x\right)^{p-3}+\left(1-\frac{n-2}{n}x\right)^{p-3}-(1+x)^{p-3}-(1-x)^{p-3}\right]\ge 0, 
\end{align*}
for all $x\in[0,1]$, by Karamata's inequality, since the function $x\mapsto x^{p-3}$ is concave. Then $\alpha$ is increasing on $[0,1]$, thus $\alpha(x)\geq0$ for all $x\in[0,1]$ as desired. 

This concludes the proof of our main theorem.
\end{proof}


\section{Acknowledgements.}
The author C. G-R. was funded by FCT/Portugal through project UIDB/04459/2020 
with DOI identifier 10-54499/UIDP/04459/2020.
J.M. was partially supported by the AMS Stefan Bergman Fellowship and the Simons Foundation Grant $\# 453576$. The second author is thankful to Jos\'e Gaitan for helpful discussion.

\bibliographystyle{amsplain}

\end{document}